\documentclass[11pt]{amsart}
\usepackage[T1]{fontenc}
\usepackage[utf8]{inputenc}
\usepackage{amsmath,amssymb,amsthm}
\usepackage{mathtools}
\usepackage{booktabs}
\usepackage[margin=1.1in]{geometry}
\usepackage{graphicx}
\usepackage{pgfplots}
\pgfplotsset{compat=1.18}
\usepgfplotslibrary{fillbetween}
\usepackage{hyperref}

\DeclareMathOperator{\rad}{rad}

\theoremstyle{plain}
\newtheorem{theorem}{Theorem}[section]
\newtheorem{proposition}[theorem]{Proposition}
\newtheorem{lemma}[theorem]{Lemma}
\newtheorem{corollary}[theorem]{Corollary}
\newtheorem{conjecture}[theorem]{Conjecture}

\theoremstyle{definition}
\newtheorem{definition}[theorem]{Definition}
\newtheorem{example}[theorem]{Example}
\newtheorem{objective}[theorem]{Objective}

\theoremstyle{remark}
\newtheorem{remark}[theorem]{Remark}

\numberwithin{equation}{section}

\begin{document}

\title{The quality plane of $abc$-triples and the decomposition into gains}

\author{R. Laniewski}
\author{K. M\"uller}

\date{October 5, 2026}



\begin{abstract}\label{abs:quality-plane}
To a coprime triple $a+b=c$ we attach two coordinates, the quality $X$ and the logarithmic mass $Y$, whose ratio $\lambda$ measures the balance of the triple and lies asymptotically between $2$ and $3$. In these coordinates the $abc$ conjecture and Szpiro's conjecture for Frey curves become asymptotic bounds with thresholds $X=1$ and $Y=3$, and the direct implications between them follow from the range of $\lambda$. We then place the approximation and power gains of M\"uller, Taktikos and de Weger in the same plane. Relative to a chosen presentation of the triple, one exponent factor $f$ relates the gains to the coordinates, and Szpiro's conjecture for this family is equivalent to the asymptotic bound $f\le 3$ uniformly over all presentations. Under the two separate asymptotic gain bounds for a fixed choice of presentations, we obtain $\limsup X\le 9/2$ and $\limsup Y\le 27/2$ as the radical tends to infinity. Stronger conclusions depend on the presentation or on a joint bound for the gains. We close with numerical data and two research objectives.
\end{abstract}

\maketitle

\section{Introduction}\label{sec:intro}

The quality of a coprime triple $a+b=c$ is the ratio $\log c/\log\rad(abc)$. The $abc$ conjecture of Masser and Oesterl\'e asserts that, for every $\varepsilon>0$, only finitely many triples have quality above $1+\varepsilon$~\cite{Oesterle1988Fermat}. The same triples also give rise to elliptic curves. The Frey curve $y^2=x(x-a)(x+b)$ has a discriminant governed by $abc$ and a conductor governed by $\rad(abc)$~\cite{Frey1986}, so that Szpiro's conjecture, which bounds the discriminant of an elliptic curve by a power of its conductor~\cite{Szpiro1990}, becomes a statement about the triple.

More recently, M\"uller and Taktikos proposed to factor the quality into two parts, an approximation gain and a power gain, for triples arising from convergents of surds, and they conjectured that the approximation gain stays below $3/2$~\cite{MullerTaktikos2026}. De Weger extended the construction to all $abc$-triples, in several variants~\cite{deWeger2026}. In the computations of de Weger, the real approximation gain stays below $3/2$ and the real power gain below $3$. Related conditional estimates appear in~\cite{Muller2026Superelliptic}.

This note places these quantities in a single plane. The quality $X$ serves as the horizontal coordinate, and the logarithmic mass $Y=\log(abc)/\log\rad(abc)$, which we also call the weak quality, as the vertical one. Their ratio $\lambda=Y/X$ depends only on the relative sizes of $a$, $b$ and $c$, and it lies asymptotically between $2$ and $3$ (Lemma~\ref{lem:lambda}). In these coordinates the $abc$ conjecture and Szpiro's conjecture for Frey curves become two thresholds, $X=1$ and $Y=3$, and the range of $\lambda$ gives the implications between them (Proposition~\ref{prop:implications}).

The gains depend on how the triple is written as a product of powers, and we make this choice explicit through the notion of a presentation (Definition~\ref{def:presentation}). Once a presentation is fixed, a single exponent factor $f$ relates the two gains to the two coordinates. The power gain is at least $1$, and Szpiro's conjecture can then be read as a bound on this exponent factor (Corollary~\ref{cor:fbound}).

We then examine the conjectured bounds on the two gains. Taken separately, they bound a product through its factors, and the resulting estimate for the quality is $9/2$, well above the threshold of the $abc$ conjecture (Proposition~\ref{prop:gainbounds}). How much more they give depends on the presentation, and on whether the two gains can be bounded jointly. Section~\ref{sec:data} collects numerical data, including computations of de Weger on the weak quality, and Section~\ref{sec:objectives} states two research objectives.

\section{Two coordinates}\label{sec:coordinates}

Throughout, a \emph{triple} is a triple $(a,b,c)$ of coprime positive integers with $a+b=c$ and $a<b$. Following the tables of de Smit~\cite{deSmitABC}, we call it an \emph{$abc$-triple} when its quality exceeds $1$.

\begin{definition}\label{def:coordinates}
For a triple $(a,b,c)$, put $N=\rad(abc)$ and define
\[
X=\frac{\log c}{\log N},\qquad Y=\frac{\log(abc)}{\log N},\qquad \lambda=\frac{\log(abc)}{\log c}.
\]
We call $X$ the \emph{quality}, $Y$ the \emph{logarithmic mass}, and $\lambda$ the \emph{balance} of the triple.
\end{definition}

The logarithms in the denominators are positive. Indeed $b$ and $c$ are coprime and larger than $1$, so that $N$ has at least two prime factors, and one of $a$, $b$, $c$ is even. Hence $N\ge 6$ and $c\ge 3$.

Only finitely many triples have a given radical, since the corresponding $S$-unit equation has finitely many solutions~\cite{Mahler1933,SilvermanAEC}. For a real quantity $Z$ attached to triples, we may therefore write
\[
\limsup_{N\to\infty}Z=\lim_{M\to\infty}\,\sup\{Z(a,b,c) : \rad(abc)\ge M\}
\]
with values in $[-\infty,+\infty]$, and similarly for the lower limit. A statement of the form \emph{for every $\varepsilon>0$, there are only finitely many triples with $Z>\kappa+\varepsilon$} is then equivalent to $\limsup_{N\to\infty}Z\le\kappa$. Since $N\le abc<c^3$, the integer $c$ tends to infinity with $N$.

The three quantities satisfy
\begin{equation}\label{eq:YlambdaX}
Y=\lambda X.
\end{equation}
We also refer to the mass $Y$ as the \emph{weak quality} of the triple. It coincides with the ratio $\rho$ used in the search for triples that are extreme for Szpiro's conjecture~\cite{Nitaj1993,Dokchitser2004}.

\begin{lemma}\label{lem:lambda}
Let $(a,b,c)$ be a triple. If $a\ge 2$, then $2<\lambda<3$. If $a=1$, then
\[
2-\frac{1}{(c-1)\log c}<\lambda<2.
\]
In particular, $2\le\liminf_{N\to\infty}\lambda\le\limsup_{N\to\infty}\lambda\le 3$. Moreover, every point of the interval $[2,3]$ is a limit of $\lambda$ along some sequence of distinct triples.
\end{lemma}

\begin{proof}
We have $\lambda=1+\log(ab)/\log c$. Since $a<c$ and $b<c$, the product $ab$ is smaller than $c^2$, whence $\lambda<3$. If $a\ge2$, then $ab\ge 2b>a+b=c$, whence $\lambda>2$. If $a=1$, then $b\ge 2$ and $c\ge 3$, and $ab=c-1$, so that
\[
\lambda=2+\frac{\log(1-1/c)}{\log c}
\]
and the bound follows from $-\log(1-1/c)<1/(c-1)$. The bounds on the lower and upper limits follow, since $c$ tends to infinity with $N$.

For the last assertion, let $c$ run through the odd primes. For $0\le\alpha<1$, put $a=\lceil c^\alpha\rceil$ and $b=c-a$. Then $a<b$ for $c$ large, the triple is coprime because $c$ is prime, and
\[
\lambda=1+\frac{\log a+\log(c-a)}{\log c}
\]
tends to $2+\alpha$. For $\alpha=1$, put $a=(c-1)/2$ and $b=(c+1)/2$, for which $\lambda$ tends to $3$.
\end{proof}

The lemma also gives a bound valid for every triple. When $a=1$, the balance
\[
\lambda=2-\frac{\log\bigl(c/(c-1)\bigr)}{\log c}
\]
increases with $c$, since the numerator decreases and the denominator increases. Every triple therefore satisfies
\[
\frac{\log 6}{\log 3}\le\lambda<3
\]
with equality on the left only for $1+2=3$, where $\lambda=1.6309\ldots$.

The balance admits a more descriptive form. Since $c/2<b<c$, one has
\begin{equation}\label{eq:balance-scale}
\lambda=2+\frac{\log a}{\log c}+\frac{\log(b/c)}{\log c}\qquad\text{with}\qquad -\frac{\log 2}{\log c}<\frac{\log(b/c)}{\log c}<0.
\end{equation}
Along a sequence with $c\to\infty$, the balance tends to $2$ precisely when $\log a/\log c\to 0$, and to $3$ precisely when $\log a/\log c\to 1$.

\section{The Frey curve and the two conjectures}\label{sec:conjectures}

\begin{conjecture}[$abc$]\label{conj:abc}
For every $\varepsilon>0$, there are only finitely many triples with $X>1+\varepsilon$.
\end{conjecture}

To a triple we attach the Frey curve
\[
E_{a,b,c}\colon y^2=x(x-a)(x+b).
\]
Let $\Delta_E$ denote its minimal discriminant and $N_E$ its conductor. The curve is semistable away from $2$, and one has
\begin{equation}\label{eq:freyinvariants}
|\Delta_E|=2^s(abc)^2\qquad\text{and}\qquad N_E=2^t\rad(abc)
\end{equation}
with integers $-8\le s\le 4$ and $-1\le t\le 7$ (see~\cite{SilvermanAEC} and~\cite{Pasten2024Shimura}). The exponents $s$ and $t$ depend on the reduction of the curve at $2$, as determined by Tate's algorithm, and account for the passage from the model above, whose discriminant is $2^4(abc)^2$, to a minimal model at $2$. The model above is minimal at every odd prime, since $c_4=16(c^2-ab)$ is prime to each odd prime dividing $abc$. A minimal model is therefore reached by a change of variables with $u=2^k$ for an integer $k\ge 0$, which divides $c_4$ by $2^{4k}$ and the discriminant by $2^{12k}$. Since $a$ and $b$ are coprime, the integer $c^2-ab=a^2+ab+b^2$ is odd, so that $v_2(c_4)=4$. An integral model after the change of variables requires $4k\le v_2(c_4)$, hence $k\le 1$, and $s=4-12k$ takes only the values $4$ and $-8$. The \emph{Szpiro ratio} of the curve is $\sigma(E)=\log|\Delta_E|/\log N_E$. The lower value $t=-1$ does occur. For the triple $1+15=16$, the curve $E_{1,15,16}$ has good reduction at $2$ and conductor $15=\rad(240)/2$.

\begin{conjecture}[Szpiro, for Frey curves]\label{conj:szpiro}
For every $\varepsilon>0$, there are only finitely many triples with $\sigma(E_{a,b,c})>6+\varepsilon$.
\end{conjecture}

\begin{lemma}\label{lem:szpiroY}
Conjecture~\ref{conj:szpiro} holds if and only if, for every $\varepsilon>0$, there are only finitely many triples with $Y>3+\varepsilon$.
\end{lemma}

\begin{proof}
By~\eqref{eq:freyinvariants}, with $N=\rad(abc)$,
\[
\sigma(E_{a,b,c})=2Y\cdot\frac{\log N}{\log N+t\log 2}+\frac{s\log 2}{\log N+t\log 2}.
\]
Since $2Y\log N=2\log(abc)$, this may be written as
\[
\frac{\sigma(E_{a,b,c})}{2Y}=\frac{\log N}{\log N+t\log 2}\left(1+\frac{s\log 2}{2\log(abc)}\right).
\]
Since $s$ and $t$ are bounded and $\log(abc)\ge\log N$, the right-hand side tends to $1$ as $N\to\infty$, that is $\sigma(E_{a,b,c})=2Y(1+o(1))$. Hence
\[
\limsup_{N\to\infty}\sigma(E_{a,b,c})=2\limsup_{N\to\infty}Y
\]
in $[0,+\infty]$. By the convention of Section~\ref{sec:coordinates}, Conjecture~\ref{conj:szpiro} is the bound $\limsup_{N\to\infty}\sigma(E_{a,b,c})\le 6$, and the assertion of the lemma is the bound $\limsup_{N\to\infty}Y\le 3$. The two are equivalent.
\end{proof}

In the plane with coordinates $(X,Y)$, Conjecture~\ref{conj:abc} places all but finitely many triples to the left of the vertical line $X=1+\varepsilon$ for every $\varepsilon>0$, and Conjecture~\ref{conj:szpiro} places them below the horizontal line $Y=3+\varepsilon$. We refer to $X=1$ and $Y=3$ as the two thresholds. The balance $\lambda$ relates the two.

\begin{proposition}\label{prop:implications}
\begin{enumerate}
\item[\textup{(i)}] Conjecture~\ref{conj:abc} implies Conjecture~\ref{conj:szpiro}.
\item[\textup{(ii)}] Conjecture~\ref{conj:szpiro} implies that, for every $\varepsilon>0$, there are only finitely many triples with $X>3/2+\varepsilon$.
\end{enumerate}
\end{proposition}

\begin{proof}
By~\eqref{eq:YlambdaX} and Lemma~\ref{lem:lambda}, one has $Y<3X$ for every triple, so that $\limsup Y\le 3\limsup X$ as $N\to\infty$. Conjecture~\ref{conj:abc} gives $\limsup X\le 1$, hence $\limsup Y\le 3$, and (i) follows from Lemma~\ref{lem:szpiroY}. For (ii), Lemma~\ref{lem:szpiroY} gives $\limsup Y\le 3$, and Lemma~\ref{lem:lambda} gives $\liminf\lambda\ge 2$. Since $X=Y/\lambda$ with $Y$ and $\lambda$ positive, one obtains
\[
\limsup_{N\to\infty}X\le\frac{\limsup_{N\to\infty}Y}{\liminf_{N\to\infty}\lambda}\le\frac32.
\]
\end{proof}

\begin{remark}\label{rem:classical}
Both implications in Proposition~\ref{prop:implications} are classical~\cite{Oesterle1988Fermat}. In the direct comparison, the constant $3/2$ in \textup{(ii)} is the quotient of the mass threshold $3$ by the lower limit $2$ of the balance, while the constant $3$ in \textup{(i)} is its upper limit. These constants arise from $Y=\lambda X$ and the limiting range of $\lambda$.

The full Szpiro conjecture over $\mathbb{Q}$ gives more, through a curve outside this family. The quotient of $E_{a,b,c}$ by the point $(0,0)$ is the curve
\[
E'\colon y^2=x\bigl(x^2-2(b-a)x+c^2\bigr)
\]
with discriminant $-2^8c^4ab$ and invariant $c_4(E')=16(a^2-14ab+b^2)=16(c^2-16ab)$. Isogenous curves have the same conductor, so $E'$ inherits from~\eqref{eq:freyinvariants} the bound $N_{E'}=N_E\le 2^7\rad(abc)$. The invariant $c_4(E')$ is prime to every odd prime dividing $abc$, so this model is minimal away from $2$. At $2$, the valuation $v_2(c_4(E'))$ equals $4$, $6$, $9$ or $8$ according as $v_2(c)$ is $0$, $1$, $2$ or at least $3$, so that a minimal model is reached with a scaling $u=2^k$ where $k\le 2$. Hence $|\Delta_{E'}|\ge 2^{-16}c^4ab>2^{-17}c^5$. Szpiro's conjecture for $E'$ with exponent $6+\eta$ then gives
\[
c^5\ll|\Delta_{E'}|\ll_\eta N_{E'}^{6+\eta}\ll\rad(abc)^{6+\eta}
\]
so that $c\ll_\eta\rad(abc)^{(6+\eta)/5}$. With $\varepsilon=\eta/5$ this reads $c\ll_\varepsilon\rad(abc)^{6/5+\varepsilon}$ for every $\varepsilon>0$, that is $\limsup X\le 6/5$. This consequence is due to Szpiro and appears in~\cite[\S~I.2, remark after Proposition~1, and \S~I.3]{Oesterle1988Fermat}. The \emph{modified} Szpiro conjecture, which bounds $\max\{|c_4|^3,c_6^2\}$ for a minimal model in place of $|\Delta_E|$, is equivalent to the $abc$ conjecture~\cite{Oesterle1988Fermat}.
\end{remark}

\section{Presentations and the gain decomposition}\label{sec:gains}

The gains of M\"uller and Taktikos are defined relative to a factorization of the triple into bases and exponents. We make this choice explicit.

\begin{definition}\label{def:presentation}
A \emph{presentation} of a triple $(a,b,c)$ is a positive integer $P$ with $P\mid abc$ and $\rad(P)=\rad(abc)$. Relative to a presentation $P$, the \emph{exponent factor}, the \emph{approximation gain} and the \emph{power gain} are
\[
f=\frac{\log(abc)}{\log P},\qquad G_a=\frac{\log c}{\log P},\qquad G_p=\frac{\log P}{\log N}
\]
where $N=\rad(abc)$.
\end{definition}

The definition admits the two extremes $P=N$ and $P=abc$, and presentations arising from the structure of a triple can reach both. The surd presentation of Example~\ref{ex:surd} for $1+3=4$, which arises from the convergent $2/1$ of $\sqrt3$ through $2^2-3\cdot1^2=1$, is $P=6=N$. The enhanced radical of de Weger equals $abc$ once the degree exceeds every exponent in $b$ and $c$ (Remark~\ref{rem:deWegerFamily}).

\begin{proposition}\label{prop:dictionary}
Let $P$ be a presentation of a triple. Then
\begin{equation}\label{eq:dictionary}
G_aG_p=X,\qquad fG_p=Y,\qquad \lambda G_a=f.
\end{equation}
Moreover,
\[
1\le f\le Y,\qquad 1\le G_p\le Y,\qquad \frac{1}{\lambda}\le G_a\le X.
\]
The presentation $P=abc$ gives $f=1$, $G_p=Y$ and $G_a=1/\lambda$. The presentation $P=N$ gives $f=Y$, $G_p=1$ and $G_a=X$.
\end{proposition}

\begin{proof}
The identities follow from the definitions by cancelling logarithms. Since $P$ divides $abc$ and is divisible by its own radical $N$, one has $N\le P\le abc$, which gives the three pairs of bounds. The two extreme presentations are obtained by substitution.
\end{proof}

The identity $fG_p=Y$, together with $G_p\ge 1$, shows that the largest exponent factor of a triple is its mass, attained at the presentation $P=N$. Szpiro's conjecture therefore admits the following form.

\begin{corollary}\label{cor:fbound}
Conjecture~\ref{conj:szpiro} holds if and only if, for every $\varepsilon>0$, there are only finitely many triples admitting a presentation with $f>3+\varepsilon$.
\end{corollary}

\begin{proof}
By Proposition~\ref{prop:dictionary}, every presentation satisfies $f\le Y$, with equality for $P=N$. The largest exponent factor over the presentations of a triple is therefore $Y$, and the claim follows from Lemma~\ref{lem:szpiroY}.
\end{proof}

The equivalence rests on the presentation $P=N$, which Definition~\ref{def:presentation} admits. A presentation with $P>N$ satisfies $f<Y$. Szpiro's conjecture therefore implies the corresponding bound on $f$ for restricted presentations, as in Objective~\ref{obj:maxf}. A converse for such a restriction would require a further argument.

\begin{example}[Surd presentations]\label{ex:surd}
Let $p/q$ be a positive convergent of an irrational surd $k^{1/s}$, with integers $k,s\ge 2$, and put $d=p^s-kq^s$. Assume $\gcd(p,k)=1$ and $pq>1$. Since $\gcd(p,q)=1$, the integers $|d|$, $kq^s$ and $p^s$ form a coprime triple in the appropriate order. The surd presentation of M\"uller and Taktikos is $P=|d|pkq$~\cite{MullerTaktikos2026}. It divides $abc=|d|kp^sq^s$ and has the same radical. Writing $A=\log(|d|k)$ and $B=\log(pq)$ gives
\begin{equation}\label{eq:surd-factor}
f=\frac{A+sB}{A+B}=1+(s-1)\frac{B}{A+B}.
\end{equation}
Thus $1<f<s$, and for fixed $s$ the factor approaches $s$ as $A/B$ tends to $0$. Its relation to the exponent $s$ depends on the relative logarithmic sizes of $|d|k$ and $pq$. The exponent $s$ of the presentation need not be the algebraic degree of $k^{1/s}$, since $x^s-k$ may be reducible, as for $4^{1/4}=\sqrt2$.

Reyssat's triple $2+3^{10}\cdot109=23^5$ arises from $23^5-109\cdot9^5=2$, with $s=5$ and $P=2\cdot23\cdot109\cdot9=45126$. One obtains
\[
f\approx 2.9903,\qquad G_a\approx 1.4628,\qquad G_p\approx 1.1142
\]
in agreement with the values in~\cite{Muller2026Superelliptic,deWeger2026}.
\end{example}

\begin{remark}\label{rem:deWegerFamily}
The real enhanced radicals of de Weger are presentations in the sense of Definition~\ref{def:presentation}. For an integer $d\ge 2$, write $b=\beta^d b'$ and $c=\gamma^d c'$ with $\beta$ and $\gamma$ as large as possible. The enhanced radical $\operatorname{rrad}_d=a\beta b'\gamma c'$ divides $abc$ and is divisible by every prime factor of $abc$, and the approximation and power gains of degree $d$ defined in~\cite{deWeger2026} are the gains $G_a$ and $G_p$ of this presentation~\cite[Section~2]{deWeger2026}.

De Weger tabulates, for Reyssat's triple, a family of presentations indexed by a degree $d$~\cite[Table~1]{deWeger2026}. The degree $d=5$ gives the surd presentation of Example~\ref{ex:surd}. For $d\ge 11$ the tabulated presentation is $P=abc$, and the tabulated values $0.48918$ and $3.33188$ are $1/\lambda$ and $Y$, in accordance with Proposition~\ref{prop:dictionary}. As the presentation varies, the gains of one triple move along the hyperbola $G_aG_p=X$, over most of the range allowed by that proposition (Figure~\ref{fig:planes}).
\end{remark}

\begin{example}[Pythagorean triples]\label{ex:pythagorean}
A primitive Pythagorean triple $x^2+y^2=z^2$, with $x$ and $y$ interchanged if necessary so that $x<y$, gives the triple $(x^2,y^2,z^2)$ and the presentation $P=xyz$. Here $f=2$ exactly, and $G_a=2/\lambda$ lies in the interval $(2/3,1)$ by Lemma~\ref{lem:lambda}. The triple $18424^2+36207^2=40625^2$ gives
\[
X\approx 1.2659,\quad Y\approx 3.6896,\quad \lambda\approx 2.9146,\quad G_a\approx 0.6862,\quad G_p\approx 1.8448.
\]
\end{example}

\section{The two gain bounds}\label{sec:gainbounds}

Fix a rule assigning to each triple $(a,b,c)$ a presentation $P(a,b,c)$, and read the gains of each triple relative to this presentation. We consider the two statements below.

\begin{enumerate}
\item[(A)] For every $\varepsilon>0$, there are only finitely many triples with $G_a>3/2+\varepsilon$.
\item[(B)] For every $\varepsilon>0$, there are only finitely many triples with $G_p>3+\varepsilon$.
\end{enumerate}

Statement \textup{(A)} is an asymptotic version of the proposed approximation bound for surd presentations~\cite{MullerTaktikos2026}. For the real gains of de Weger, the numerical motivation for the pair concerns the presentation minimizing $\operatorname{rrad}_d$ over $d\ge 2$, which maximizes $G_a$ and minimizes $G_p$~\cite[Definition~4 and Section~2.4]{deWeger2026}. We study \textup{(A)} and \textup{(B)} as hypotheses for the chosen presentations.

\begin{proposition}\label{prop:gainbounds}
Assume \textup{(A)} and \textup{(B)} for a fixed rule of presentations. Then, for every $\varepsilon>0$, there are only finitely many triples with $X>9/2+\varepsilon$, and only finitely many with $Y>27/2+\varepsilon$.
\end{proposition}

\begin{proof}
By~\eqref{eq:dictionary}, one has $X=G_aG_p$ and $Y=\lambda G_aG_p$, and Lemma~\ref{lem:lambda} gives $\lambda<3$ for every triple. Given $\varepsilon>0$, choose $\delta>0$ with $3(3/2+\delta)(3+\delta)<27/2+\varepsilon$, which also gives $(3/2+\delta)(3+\delta)<9/2+\varepsilon$. By \textup{(A)} and \textup{(B)}, all triples outside a finite set satisfy $G_a\le 3/2+\delta$ and $G_p\le 3+\delta$. For these, $X<9/2+\varepsilon$ and $Y<27/2+\varepsilon$.
\end{proof}

\begin{remark}\label{rem:plane}
The strict inequalities at the thresholds of \textup{(A)} and \textup{(B)} take the form
\[
G_a<\tfrac32\ \Longleftrightarrow\ Y>\tfrac23 fX\qquad\text{and}\qquad G_p<3\ \Longleftrightarrow\ Y<3f.
\]
For fixed $f$, their closure in the non-negative quadrant is the triangle with vertices $(0,0)$, $(0,3f)$ and $(9/2,3f)$. In the coordinate $Y/f=G_p$, they become $\tfrac23 X<G_p<3$, and eliminating $G_p$ gives $X<9/2$. Equivalently, since $Y=\lambda X$ and $f=\lambda G_a$, they confine the exponent factor to the interval
\[
\frac{\lambda X}{3}<f<\frac{3\lambda}{2}
\]
which is non-empty only when $X<9/2$. This triangle describes the inequalities before the further restrictions of Proposition~\ref{prop:dictionary} and the arithmetic of presentations are imposed. Statements \textup{(A)} and \textup{(B)} allow finitely many exceptions at every positive tolerance, and their asymptotic consequence is Proposition~\ref{prop:gainbounds}.
\end{remark}

\begin{remark}\label{rem:loss}
Multiplying the constants in \textup{(A)} and \textup{(B)} gives $9/2$, whereas Conjecture~\ref{conj:abc} asks for $1$. The resulting mass bound $27/2$ likewise exceeds the threshold $3$ in Lemma~\ref{lem:szpiroY}. These comparisons describe the limits of the separate estimates. Stronger consequences may follow from the chosen presentation, as in Remark~\ref{rem:presentation}.

More generally, bounds $G_a\le C_a+\varepsilon$ and $G_p\le C_p+\varepsilon$, each with finitely many exceptions for every $\varepsilon>0$, that is, the asymptotic bounds $\limsup G_a\le C_a$ and $\limsup G_p\le C_p$ as $N\to\infty$, give $\limsup X\le C_aC_p$ as $N\to\infty$. This product estimate yields Conjecture~\ref{conj:abc} when $C_aC_p\le 1$. Since $G_p\ge 1$, such constants satisfy $C_p\ge 1$ and hence $C_a\le 1$. Individual values above these constants remain compatible with an asymptotic bound. For comparison, the largest real approximation gain computed by de Weger is $1.4628$, attained by Reyssat's triple~\cite[Section~2.4]{deWeger2026}.

The constant $3/2$ in \textup{(A)} also has a direct reading in the plane. Every presentation satisfies $G_a\le X$ by Proposition~\ref{prop:dictionary}. Conjecture~\ref{conj:abc} therefore gives $\limsup G_a\le 1$, and Conjecture~\ref{conj:szpiro} gives $\limsup G_a\le 3/2$ by Proposition~\ref{prop:implications}\textup{(ii)}, whatever the choice of presentations. Statement \textup{(A)} is thus a consequence of Szpiro's conjecture for Frey curves.
\end{remark}

\begin{remark}\label{rem:presentation}
The strength of the pair \textup{(A)} and \textup{(B)} depends on the presentation. For $P=abc$, statement \textup{(A)} holds since $G_a=1/\lambda<1$, and \textup{(B)} is equivalent to Conjecture~\ref{conj:szpiro} by Lemma~\ref{lem:szpiroY}. For $P=N$, statement \textup{(B)} holds since $G_p=1$, and \textup{(A)} is the bound $\limsup X\le 3/2$ of Proposition~\ref{prop:implications}\textup{(ii)}. Any conjecture about the gains must therefore specify its presentation.

At both extremes, the pair gives $\limsup X\le 3/2$. For $P=N$ this follows directly from $G_a=X$. For $P=abc$ it follows from $X=Y/\lambda$ and $\liminf\lambda\ge 2$. The general constant $9/2$ in Proposition~\ref{prop:gainbounds} comes from multiplying the separate bounds without these further identities. Its attainment by triples is not asserted.
\end{remark}

\section{Numerical data}\label{sec:data}

Table~\ref{tab:quality} lists the five highest-quality triples in the cited tables of de Smit~\cite{deSmitABC,BrowkinBrzezinski1994,deWeger1987}, together with Nitaj's triple of mass $4.4190\ldots$ discussed in~\cite{Nitaj1993,Dokchitser2004}. Reyssat's triple is discussed in~\cite{deWeger2026}. The values in Tables~\ref{tab:quality}, \ref{tab:reyssat} and~\ref{tab:balanced} are computed directly from the definitions with forty significant digits and rounded to four decimal places. Additional digits appear where a threshold or a plotted point calls for them.

\begin{table}[ht]
\centering
\begin{tabular}{llrrr}
\toprule
Triple & Found by & $X$ & $Y$ & $\lambda$\\
\midrule
$2+3^{10}\cdot109=23^5$ & Reyssat & $1.6299$ & $3.3319$ & $2.0442$\\
$11^2+3^2\cdot5^6\cdot7^3=2^{21}\cdot23$ & de Weger & $1.6260$ & $3.6928$ & $2.2711$\\
$19\cdot1307+7\cdot29^2\cdot31^8=2^8\cdot3^{22}\cdot5^4$ & Browkin, Brzezi\'nski & $1.6235$ & $3.7014$ & $2.2799$\\
$283+5^{11}\cdot13^2=2^8\cdot3^8\cdot17^3$ & Browkin, Brzezi\'nski, Nitaj & $1.5808$ & $3.5523$ & $2.2472$\\
$1+2\cdot3^7=5^4\cdot7$ & de Weger & $1.5679$ & $3.1357$ & $2.0000$\\
\midrule
$13\cdot19^6+2^{30}\cdot5=3^{13}\cdot11^2\cdot31$ & Nitaj & $1.5270$ & $4.4190$ & $2.8939$\\
\bottomrule
\end{tabular}

\smallskip
\footnotesize The balance of $1+2\cdot3^7=5^4\cdot7$ is $1.99997\ldots$, strictly below $2$, in accordance with Lemma~\ref{lem:lambda}. The mass of $11^2+3^2\cdot5^6\cdot7^3=2^{21}\cdot23$ is $3.6927502\ldots$, which lies within $3\cdot10^{-7}$ of a rounding boundary.
\medskip
\caption{Quality, mass and balance for extremal triples.}\label{tab:quality}
\end{table}

The orderings by quality and by mass differ. Reyssat's triple, first by quality, comes fourth of these five by mass. Nitaj's triple has lower quality than all five, and its mass is larger because its balance is close to $3$, as~\eqref{eq:YlambdaX} leads one to expect.

Table~\ref{tab:reyssat} gives the gains of Reyssat's triple under three presentations, and Figure~\ref{fig:planes} places these data in the two planes. The surd presentation lies strictly inside the range of Proposition~\ref{prop:dictionary}.

\begin{table}[ht]
\centering
\begin{tabular}{lrrrr}
\toprule
Presentation & $P$ & $f$ & $G_a$ & $G_p$\\
\midrule
$P=N$ & $15042$ & $3.3319$ & $1.6299$ & $1.0000$\\
surd, $s=5$ & $45126$ & $2.9903$ & $1.4628$ & $1.1142$\\
$P=abc$ & $2\cdot3^{10}\cdot23^5\cdot109$ & $1.0000$ & $0.4892$ & $3.3319$\\
\bottomrule
\end{tabular}
\medskip
\caption{Gains of Reyssat's triple under three presentations.}\label{tab:reyssat}
\end{table}

\begin{figure}[!htbp]
\centering
\begin{tikzpicture}
\begin{axis}[
  width=0.49\textwidth, height=0.49\textwidth,
  xmin=0, xmax=2.1, ymin=0, ymax=5,
  axis lines=left, xlabel={$X$}, ylabel={$Y$},
  tick label style={font=\scriptsize}, label style={font=\small},
  xtick={0,0.5,1,1.5}, ytick={0,1,2,3,4,5}]
\addplot[name path=lower, draw=gray, domain=0:2.1] {2*x};
\addplot[name path=upper, draw=gray, domain=0:2.1] {3*x};
\addplot[gray!15] fill between[of=lower and upper];
\addplot[dashed, thick] coordinates {(1,0) (1,5)};
\addplot[dotted, thick] coordinates {(1.5,0) (1.5,5)};
\addplot[dashed, thick] coordinates {(0,3) (2.1,3)};
\addplot[only marks, mark=*, mark size=1.4pt] coordinates {
  (1.6299,3.3319) (1.6260,3.6928) (1.6235,3.7014)
  (1.5808,3.5523) (1.5679,3.1357) (1.5270,4.4190)};
\node[font=\scriptsize, anchor=west] at (axis cs:1.66,3.2) {Reyssat};
\node[font=\scriptsize, anchor=east] at (axis cs:1.47,4.42) {Nitaj};
\node[font=\scriptsize, anchor=south east] at (axis cs:0.98,0.1) {$X=1$};
\node[font=\scriptsize, anchor=south west] at (axis cs:1.52,0.1) {$X=3/2$};
\node[font=\scriptsize, anchor=south west] at (axis cs:0.05,3.02) {$Y=3$};
\node[font=\scriptsize] at (axis cs:0.55,1.35) {$\lambda\in[2,3]$};
\end{axis}
\end{tikzpicture}\hfill
\begin{tikzpicture}
\begin{axis}[
  width=0.49\textwidth, height=0.49\textwidth,
  xmin=0, xmax=2.1, ymin=0, ymax=3.6, clip=false,
  axis lines=left, xlabel={$G_a$}, ylabel={$G_p$},
  tick label style={font=\scriptsize}, label style={font=\small},
  xtick={0,0.5,1,1.5}, ytick={0,1,2,3}]
\addplot[dashed, thick] coordinates {(0,3) (1.5,3) (1.5,0)};
\addplot[domain=0.28:1.95, samples=120] {1/x};
\addplot[domain=0.46:1.95, samples=120, densely dotted, thick] {1.63/x};
\addplot[domain=1.25:1.95, samples=60, dash dot, thick] {4.5/x};
\addplot[only marks, mark=o, mark size=1.4pt] coordinates {
  (0.77289,2.10883) (0.81715,1.99461) (0.97679,1.66863)
  (0.59037,2.76079) (0.61585,2.64658) (0.64363,2.53236)
  (0.67403,2.41814) (0.70744,2.30392) (0.48918,3.33188)
  (1.62991,1.00000)};
\addplot[only marks, mark=*, mark size=1.6pt] coordinates {(1.46283,1.11421)};
\addplot[only marks, mark=square*, mark size=1.4pt] coordinates {(0.6862,1.8448)};
\addplot[only marks, mark=triangle*, mark size=1.8pt] coordinates {(0.45020,2.94376)};
\addplot[only marks, mark=x, mark size=3pt, thick] coordinates {(1.5,3)};
\node[font=\scriptsize, anchor=west] at (axis cs:1.96,2.31) {$9/2$};
\node[font=\scriptsize, anchor=west] at (axis cs:1.96,0.51) {$1$};
\node[font=\scriptsize, anchor=west] at (axis cs:1.96,0.83) {$1.63$};
\node[font=\scriptsize, anchor=west] at (axis cs:1.47,1.28) {surd};
\end{axis}
\end{tikzpicture}
\caption{Left, the quality plane, with the cone $2X\le Y\le 3X$ of Lemma~\ref{lem:lambda}, which contains every triple with $a\ge2$, the $abc$ threshold $X=1$, the threshold $X=3/2$ of Proposition~\ref{prop:implications}\textup{(ii)}, the Szpiro threshold $Y=3$, and the six triples of Table~\ref{tab:quality}. Right, the gain plane, with the threshold box of statements \textup{(A)} and \textup{(B)}, and the hyperbolas $G_aG_p=1$, $G_aG_p=1.63$ (dotted) and $G_aG_p=9/2$ (dash-dotted, through the corner of the box), each labelled at its right end by the value of the product. The circles are the presentations of Reyssat's triple from~\cite[Table~1]{deWeger2026} together with $P=N$, the filled disc its surd presentation, the square the Pythagorean triple of Example~\ref{ex:pythagorean}, the triangle the triple attaining the largest real power gain in the computations of de Weger~\cite{deWeger2026}, and the cross the corner $(3/2,3)$ of the box.}\label{fig:planes}
\end{figure}
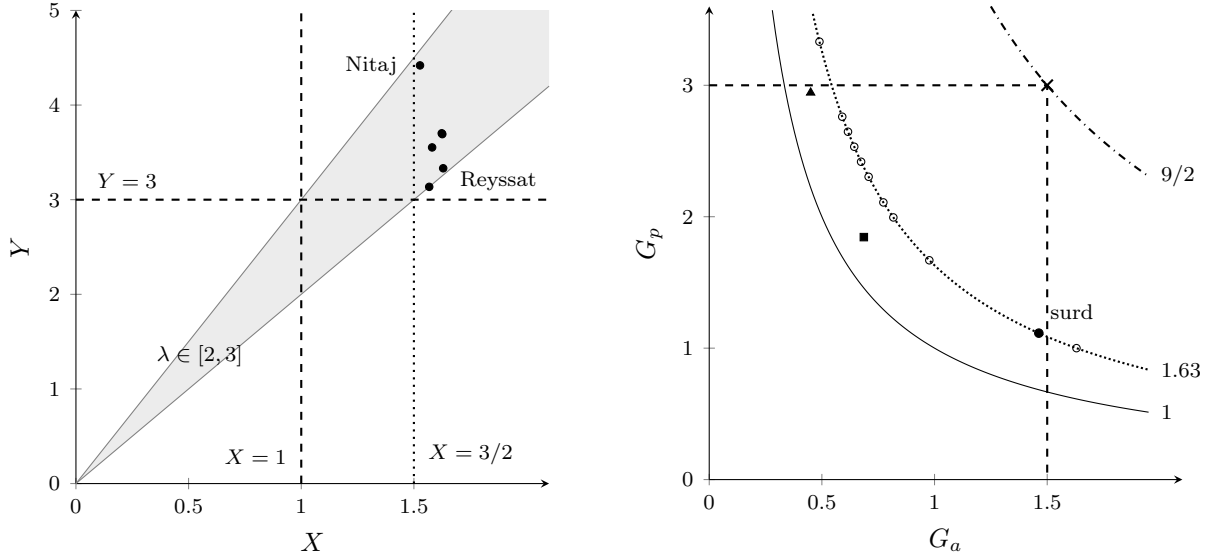

Figure~\ref{fig:weakquality} shows the quality and the weak quality $Y$ for a sample of $abc$-triples supplied by de Weger (private communication). The points with $a\ge 2$ lie in the cone $2X\le Y\le 3X$ of Lemma~\ref{lem:lambda}. Points with $a=1$ lie just below its lower edge when $c$ is large. The highest point shown is Nitaj's triple of Table~\ref{tab:quality}. Table~\ref{tab:balanced} lists the ten triples of largest balance in the supplied sample. Their proximity to $Y=3X$ reflects the logarithmic relation~\eqref{eq:balance-scale}, which allows $a/c\to 0$ while the balance tends to $3$. For these ten triples, the quotient $c/a$ lies between $2$ and $7.7$, so that the three terms also have comparable size, and the weak quality exceeds $4$ while the quality stays below $1.42$.

\begin{figure}[!htbp]
\centering
\includegraphics[width=0.85\textwidth]{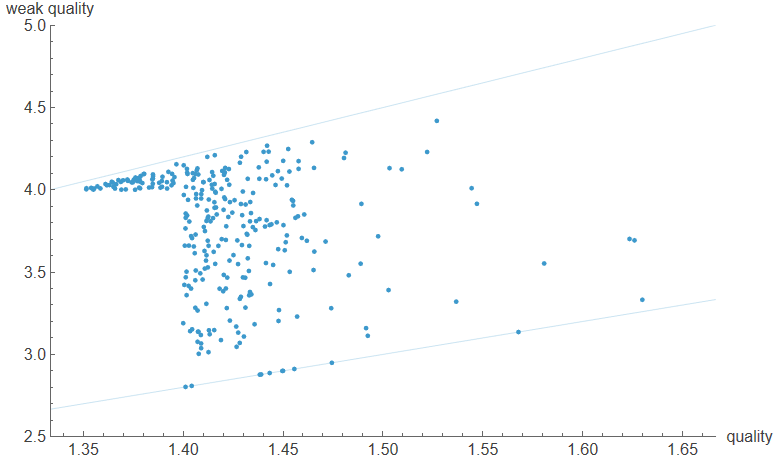}
\caption{Quality $X$ and weak quality $Y$ for a sample of $abc$-triples supplied by de Weger. The lines $Y=2X$ and $Y=3X$ bound the cone for $a\ge 2$. Triples with $a=1$ lie below its lower edge.}\label{fig:weakquality}
\end{figure}

\begin{table}[!htbp]
\centering
{\footnotesize
\begin{tabular}{lrrrr}
\toprule
Triple & $X$ & $Y$ & $\lambda$ & $c/a$\\
\midrule
$2^{14}\cdot 11^{8}\cdot 29^{3}\cdot 1123\cdot 3631^{2}+3^{8}\cdot 17^{4}\cdot 43^{2}\cdot 163^{7}=7^{28}\cdot 9491$ & $1.3966$ & $4.1552$ & $2.9752$ & $3.4425$\\
$2^{37}\cdot 3^{12}\cdot 9109^{3}+5^{13}\cdot 13^{15}\cdot 2939=7^{23}\cdot 11\cdot 793345871$ & $1.4121$ & $4.2002$ & $2.9745$ & $4.3265$\\
$2^{7}\cdot 5^{4}\cdot 7^{22}+19^{4}\cdot 37\cdot 47^{4}\cdot 53^{6}=3^{14}\cdot 11\cdot 13^{9}\cdot 191\cdot 7829$ & $1.4158$ & $4.2102$ & $2.9737$ & $2.6673$\\
$5^{13}\cdot 7^{12}\cdot 11^{8}\cdot 13+2^{13}\cdot 3^{30}\cdot 113\cdot 2297\cdot 712573=17^{3}\cdot 19^{11}\cdot 43^{3}\cdot 53^{4}$ & $1.3676$ & $4.0590$ & $2.9681$ & $7.6256$\\
$5^{6}\cdot 257^{10}+3^{3}\cdot 7\cdot 11^{6}\cdot 31\cdot 41^{5}\cdot 113^{2}\cdot 8342669=2^{77}\cdot 17^{3}\cdot 199$ & $1.3642$ & $4.0489$ & $2.9678$ & $7.5224$\\
$17^{6}\cdot 23\cdot 59^{4}\cdot 449+3^{16}\cdot 5^{10}\cdot 97^{2}=2^{31}\cdot 1481^{3}$ & $1.3805$ & $4.0968$ & $2.9676$ & $2.3095$\\
$2^{16}\cdot 5^{3}\cdot 7^{9}\cdot 11^{5}\cdot 23\cdot 479\cdot 997+3^{17}\cdot 17^{5}\cdot 19^{2}\cdot 61^{6}=13\cdot 103^{11}\cdot 149^{2}$ & $1.3802$ & $4.0953$ & $2.9673$ & $6.8317$\\
$13^{4}\cdot 17\cdot 227^{2}\cdot 257^{3}+2^{12}\cdot 3^{20}\cdot 31^{3}=5^{19}\cdot 29^{2}\cdot 53$ & $1.3516$ & $4.0094$ & $2.9664$ & $2.0018$\\
$23^{5}\cdot 43^{5}\cdot 397+3^{33}\cdot 5^{3}=2^{27}\cdot 7\cdot 29\cdot 73^{3}\cdot 101$ & $1.3611$ & $4.0348$ & $2.9644$ & $2.8499$\\
$5^{9}\cdot 7^{2}\cdot 11^{8}+2^{5}\cdot 19^{9}\cdot 2819=3\cdot 13^{2}\cdot 31^{5}\cdot 43^{4}$ & $1.3539$ & $4.0119$ & $2.9631$ & $2.4189$\\
\bottomrule
\end{tabular}}
\medskip
\caption{The ten triples of largest balance in the sample supplied by de Weger, ordered by decreasing balance.}\label{tab:balanced}
\end{table}

\section{Research objectives}\label{sec:objectives}

\begin{objective}\label{obj:maxf}
What is the largest exponent factor $f$ among the known triples, for surd presentations or for the real enhanced radical of de Weger minimized over the degree? For the latter, put $m_\ell=v_\ell(bc)$ and $M=\max_{\ell\mid bc}m_\ell$. Then the minimizing presentation is
\begin{equation}\label{eq:optimal-real-presentation}
P_*=\min_{2\le d\le\max\{2,M\}}\Bigl(a\prod_{\ell\mid bc}\ell^{\lfloor m_\ell/d\rfloor+(m_\ell\bmod d)}\Bigr).
\end{equation}
Here $\ell$ runs through the prime divisors of $bc$, and $m_\ell\bmod d$ denotes the least non-negative remainder. The exponent in the product follows from extracting the largest $d$th power, while every degree above $M$ gives $abc$, so the search is finite for each triple. Among the triples in Tables~\ref{tab:quality} and~\ref{tab:balanced}, the largest value of $\log(abc)/\log P_*$ is $2.9903\ldots$, attained by Reyssat's triple at degree $5$.

The restriction on presentations is essential to this question. Over all presentations, the maximum is $Y$ by Proposition~\ref{prop:dictionary}. Under Conjecture~\ref{conj:szpiro}, Corollary~\ref{cor:fbound} bounds $f$ asymptotically by $3$ for every presentation. For surd presentations, equation~\eqref{eq:surd-factor} gives
\begin{equation}\label{eq:surd-asymptotic-share}
\frac{\log(pq)}{\log P}\le\frac{2+o(1)}{s-1}
\end{equation}
as $N\to\infty$, uniformly in $s$. This restricts the surd presentations with exponent $s\ge 4$. For Reyssat's triple, with $s=5$, the ratio is $0.4976$ against $2/(s-1)=0.5$.
\end{objective}

\begin{objective}\label{obj:jointregion}
For a fixed type of presentation, what joint bound on $(G_a,G_p)$ is supported by the known triples? A non-increasing empirical envelope $G_p\le\varphi(G_a)$ would permit a comparison with the separate bounds $G_a<3/2$ and $G_p<3$. For an asymptotic conclusion, one would need, for every $\varepsilon>0$, a joint bound giving $G_aG_p\le 1+\varepsilon$ outside finitely many triples. By~\eqref{eq:dictionary}, this is equivalent to Conjecture~\ref{conj:abc}. The triples in Table~\ref{tab:quality} have quality below $1.63$, so their presentations lie below $G_aG_p=1.63$, far from the corner $(3/2,3)$ where the product is $9/2$.
\end{objective}

\section*{Acknowledgments}\label{sec:acknowledgments}

The authors thank Benne de Weger, Ingo Alth\"ofer and Abderrahmane Nitaj for helpful ideas, calculations and discussions.

\bibliographystyle{smfalpha}
\bibliography{quality_plane}

\end{document}